\documentclass[reqno]{amsart}

\usepackage{amsmath}
\usepackage{amsfonts}
\usepackage{amssymb,enumerate}
\usepackage{amsthm}
\usepackage[all]{xy}
\usepackage{rotating}
\usepackage{hyperref}
\usepackage{color}

\theoremstyle{plain}
\newtheorem{lem}{Lemma}[section]
\newtheorem{cor}[lem]{Corollary}
\newtheorem{prop}[lem]{Proposition}
\newtheorem{thm}[lem]{Theorem}

\theoremstyle{definition}

\newtheorem{ex}[lem]{Example}
\newtheorem{question}[lem]{Question}

\newtheorem{disc}[lem]{Remark}

\newtheorem{fact}[lem]{Fact}

\newtheorem*{convention*}{Convention}
\newtheorem*{mthm*}{Main Theorem}

\newcommand{\catd}{\mathcal{D}}

\newcommand{\pd}{\operatorname{pd}}	
\newcommand{\gdim}{\mathrm{G}\text{-}\!\dim}

\newcommand{\id}{\operatorname{id}}

\newcommand{\depth}{\operatorname{depth}}

\newcommand{\ann}{\operatorname{Ann}}

\newcommand{\soc}{\operatorname{Soc}}

\newcommand{\HH}{\operatorname{H}}

\newcommand{\im}{\operatorname{Im}}
\newcommand{\shift}{\mathsf{\Sigma}}

\newcommand{\Ker}{\operatorname{Ker}}

\newcommand{\ideal}[1]{\mathfrak{#1}}
\newcommand{\m}{\ideal{m}}

\newcommand{\p}{\ideal{p}}
\newcommand{\q}{\ideal{q}}
\newcommand{\fm}{\ideal{m}}

\newcommand{\ol}{\overline}
\newcommand{\wti}{\widetilde}

\newcommand{\ass}{\operatorname{Ass}}

\newcommand{\bbz}{\mathbb{Z}}

\newcommand{\xra}{\xrightarrow}
\newcommand{\xla}{\xleftarrow}

\renewcommand{\geq}{\geqslant}
\renewcommand{\leq}{\leqslant}
\renewcommand{\ker}{\Ker}

\newcommand{\Ext}[4][R]{\operatorname{Ext}_{#1}^{#2}(#3,#4)}	
\newcommand{\Rhom}[3][R]{\mathbf{R}\!\operatorname{Hom}_{#1}(#2,#3)}	
\newcommand{\Lotimes}[3][R]{#2\otimes^{\mathbf{L}}_{#1}#3}
\newcommand{\Otimes}[3][R]{#2\otimes_{#1}#3}
\newcommand{\Hom}{\operatorname{Hom}}	
\newcommand{\Tor}[4][R]{\operatorname{Tor}^{#1}_{#2}(#3,#4)}

\def\Tor{\operatorname{Tor}}
\def\Ext{\operatorname{Ext}}

\numberwithin{equation}{lem}

\begin{document}

\bibliographystyle{amsplain}

\title{Vanishing of Ext and Tor over fiber products}

\author{Saeed Nasseh}

\address{Saeed Nasseh,
Department of Mathematical Sciences,
Georgia Southern University,
Statesboro, Georgia 30460, USA}

\email{snasseh@georgiasouthern.edu}
\urladdr{https://cosm.georgiasouthern.edu/math/saeed.nasseh}

\author{Sean Sather-Wagstaff}

\address{Sean Sather-Wagstaff,
Department of Mathematical Sciences,
Clemson University,
O-110 Martin Hall, Box 340975,
Clemson, S.C. 29634,
USA}

\email{ssather@clemson.edu}
\urladdr{https://ssather.people.clemson.edu/}

\thanks{Sather-Wagstaff  was supported in part by North Dakota EPSCoR,
National Science Foundation Grant EPS-0814442,
and NSA grant H98230-13-1-0215.}





\begin{abstract}
Consider a non-trivial fiber product $R=S\times_kT$ of local rings $S$, $T$ with common residue field $k$. Given two finitely generate $R$-modules $M$ and $N$, we show that if $\operatorname{Tor}^R_i(M,N)=0=\operatorname{Tor}^R_{i+1}(M,N)$ for some $i\geq 5$, then $\operatorname{pd}_R(M)\leq 1$ or $\operatorname{pd}_R(N)\leq 1$. From this, we deduce several consequence, for instance, that $R$ satisfies the Auslander-Reiten Conjecture.
\end{abstract}

\maketitle

\section{Introduction}\label{sec160319a}

Throughout this paper, let $(S,\fm_S,k)$ and $(T,\fm_T,k)$ be commutative  local (meaning local and noetherian) rings.
Let $S\xra{\pi_S} k\xla{\pi_T}T$ denote the natural surjections onto the common residue field, and
assume that $S\neq k\neq T$. Throughout, let $R$ denote the fiber product (i.e., the pull-back) ring
$$R:=S\times_kT=\left\{(s,t)\in S\times T\mid \pi_S(s)=\pi_T(t)\right\}.$$
It is also
commutative and local with maximal ideal $\m_R=\m_S\oplus\m_T$ and residue field $k$.
And it is universal with respect to the next
commutative diagram.
\begin{equation}\label{eq120222a}\notag
\begin{split}
\xymatrix{
R\ar[r]\ar[d]&T\ar[d]^{\pi_T}\\
S\ar[r]^{\pi_S}&k
}
\end{split}\end{equation}
Also, the subsets $\m_S,\m_T\subseteq R$ are ideals of $R$, and
we have  ring isomorphisms $S\cong R/\frak m_T$ and $T\cong R/\frak m_S$.

Many results about $R$ state that its properties reflect those of the rings $S$ and $T$.
For instance, Dress and Kr\"amer~\cite[Remark~3]{dress} show
that for every finitely generated $R$-module $N$, the second syzygy of $N$ decomposes as a direct sum
$N'\cong N_1\oplus N_2$ where $N_1$ is a finitely generated $S$-module
and $N_2$ is a finitely generated $T$-module.
Other examples of this include work of Moore~\cite[Theorem~1.8]{moore} which shows how, given a finitely generated $S$-module $M_1$,
one can use the
minimal $S$-free resolution of $M_1$ with the minimal free resolutions of $k$ over $S$ and $T$ to obtain the
minimal $R$-free resolution of $M_1$. (Fact~\ref{fact160316a} below describes part of this construction.)
See, e.g., the papers of Kostrikin and {\v{S}}afarevi{\v{c}}~\cite{kostrikin} and Lescot~\cite{lescot:sbpfal} for more
results in this theme.
The point here is that the module category of $R$ is deeply related to the module categories of $S$ and $T$.

However, there are glaring counterpoints to this theme. For instance,
one consequence of Lescot's work is the equality
\begin{equation}\label{eq160315a}\tag{$\ast$}
\depth(R)=\min\{\depth(S),\depth(T),1\}.
\end{equation}
(See, e.g., Christensen, Striuli, and Veliche~\cite[(3.2.1)]{christensen:gmirlr}.)
Technically, this describes the depth of $R$ in terms of the depths of $S$ and $T$,
but it shows, for instance, that $R$ is almost never Cohen-Macaulay, even when $S$ and $T$ are so.

The  results of this paper  give further counterpoints to this theme.
For instance, the next result is contained in Theorems~\ref{thm160316a} and~\ref{thm160316b}.
In the language of Celikbas and Sather-Wagstaff~\cite{celikbas:tgp} this result says that every $R$-module of
infinite projective dimension is a ``pd-test module''.

\begin{thm}\label{thm160319a}
Let $M$ and $N$ be finitely generated modules over the fiber product~$R$.
\begin{enumerate}[\rm(a)]
\item \label{thm160319a1}
If $\depth(S)=0$ or $\depth(T)=0$ and
$\Tor^R_i(M,N)=0$ for some $i\geq 5$,
then $M$ or $N$ is free over $R$.
\item \label{thm160319a2}
In general, if
$\Tor^R_i(M,N)=0=\Tor^R_{i+1}(M,N)$ for some $i\geq 5$, then
$\pd_R(M)\leq 1$ or $\pd_R(N)\leq 1$.
\end{enumerate}
\end{thm}

This result gives another counterpoint to the above theme because,
even if $S$ and $T$ have Tor-independent modules of infinite projective dimension
(e.g., if $S=k[U,V,X,Y]/((U,V)^2+(X,Y)^2)\cong T$),
the theorem shows that $R$ will not have such modules.
Section~\ref{sec160319b} is primarily devoted to the proof of this result.
In the subsequent Section~\ref{sec160325a}, we explore the consequences for ``depth formulas'' over $R$;
some of these are expected, others are surprising to us.

The final Section~\ref{sec160319d} documents Ext-vanishing results that follow from Theorem~\ref{thm160319a}.
For instance, the next result, contained in Theorem~\ref{cor160316b}, shows that the fiber product $R$
satisfies the Auslander-Reiten Conjecture, regardless of whether we know it for $S$ or $T$.

\begin{thm}\label{thm160319b}
Let $M$ be a finitely generated module  over the fiber product $R$.
If $\Ext_R^i(M,M\oplus R)=0$ for $i\geq 1$, then $M$ is $R$-free.
\end{thm}

\section{Tor-Vanishing}\label{sec160319b}

In this section we prove Theorem~\ref{thm160319a} from the introduction,
beginning with some preliminary facts.
Recall that $R$ is a non-trivial fiber product, as described above.

\begin{fact}[\protect{\cite[Remark~3]{dress}}]\label{fact160316b}
For every finitely generated $R$-module $N$, the second syzygy of $N$ decomposes as a direct sum
$N'\cong N_1\oplus N_2$ where $N_1$ is a finitely generated $S$-module
and $N_2$ is a finitely generated $T$-module; in other words, $N_1$
is a finitely generated $R$-module annihilated by $\m_T$, and similarly for $N_2$.
Moreover, the proof of~\cite[Remark~3]{dress} shows the following.
The syzygy $N'$ is a submodule of a finite-rank free $R$-module $R^n$, as the image of an $R$-linear map $f\colon R^m\to R^n$,
and we have $N_2=\im(f)\cap\m_TR^n\subseteq\m_TR^n\cong\m_TT^n$, and similarly for $N_1$.
In particular, $N_2$ is a first syzygy over $T$ and $N_1$ is a first syzygy over $S$.
\end{fact}

\begin{fact}[\protect{\cite[Theorem~1.8]{moore}}]
\label{fact160316a}
Let $M_1$ be a finitely generated $S$-module.
We describe part of a minimal $R$-free resolution of $M_1$, in terms of the following:
\begin{enumerate}[(1)]
\item \label{lem160316a1}
Let $S^{\beta_2}\xra{d_2} S^{\beta_1}\xra{d_1} S^{\beta_0}$ be the beginning of a minimal $S$-free resolution of $M_1$.
(In particular, $\beta_i$ is the $i$th Betti number $\beta_i^S(M_1)$ of $M_1$ over $S$.)
\item \label{lem160316a2}
Let $S^{b_1}\xra {f_1}S$ be a minimal $S$-free presentation of $k$ where $f_1$ is a $1\times b_1$ matrix
whose entries minimally generate $\m_S$.
(In particular, $b_1=\beta_1^S(k)$.)
\item \label{lem160316a3}
Let $T^{c_2}\xra{g_2}T^{c_1}\xra{g_1}T$ be the beginning of a minimal $T$-free resolution of $k$ where $g_1$ is a $1\times c_1$ matrix
whose entries minimally generate $\m_T$.
(In particular, $c_1=\beta_1^T(k)$.)
\end{enumerate}
Then a minimal $R$-free resolution of $M_1$ begins as follows.
\begin{equation}\label{eq160316a}
R^{\beta_2}\oplus R^{c_1\beta_1}\oplus R^{b_1c_1\beta_0}\oplus R^{c_2\beta_0}
\xra{\left[\begin{smallmatrix}\wti{d_2} & \widehat{g_1}&0&0 \\ 0&0&\widehat{f_1} & \ol{g_2} \end{smallmatrix}\right]}
R^{\beta_1}\oplus R^{c_1\beta_0}
\xra{\left[\wti{d_1} \ \ \ol{g_1}\right]}
R^{\beta_0}
\end{equation}
Here each entry in each matrix is induced by the corresponding map from~\eqref{lem160316a1}--\eqref{lem160316a3} above.
For instance, $\wti{d_2}$ uses the same matrix as $d_2$, only considered over $R$; in particular, these entries are in $\m_S$.
And $\widehat{g_1}$ uses $\beta_1$-many copies of the matrix for $g_1$; in particular, these entries are in $\m_T$.
\end{fact}

Our first lemma is akin to~\cite[Lemma~3.2]{nasseh:oeire}, though our proof is vastly different.

\begin{lem}\label{lem160316a}
Let $M_1$, $N_1$ be finitely generated $S$-modules, and let $M_2$, $N_2$ be finitely generated $T$-modules.
Then there are isomorphisms  over the fiber product $R$.
\begin{align*}
\Tor^R_1(M_1,N_1)
&\cong\Tor^S_1(M_1,N_1)\oplus\left(\frac{N_1}{\m_SN_1}\right)^{\beta^T_1(k)\beta^S_0(M_1)}\\
&\cong\Tor^S_1(M_1,N_1)\oplus\left(\frac{M_1}{\m_SM_1}\right)^{\beta^T_1(k)\beta^S_0(N_1)}\\
\Tor^R_1(M_2,N_2)
&\cong\Tor^T_1(M_2,N_2)\oplus\left(\frac{N_2}{\m_TN_2}\right)^{\beta^S_1(k)\beta^T_0(M_2)}\\
&\cong\Tor^T_1(M_2,N_2)\oplus\left(\frac{M_2}{\m_TM_2}\right)^{\beta^S_1(k)\beta^T_0(N_2)}\\
\Tor^R_1(M_1,N_2)
&\cong\Tor^T_1(k,N_2)^{\beta^S_0(M_1)}\oplus\left(\frac{N_2}{\m_TN_2}\right)^{\beta^S_1(M_1)}\\
&\cong\Tor^S_1(M_1,k)^{\beta^T_0(N_2)}\oplus\left(\frac{M_1}{\m_SM_1}\right)^{\beta^T_1(N_2)}\\
\Tor^R_1(M_2,N_1)
&\cong\Tor^T_1(k,M_2)^{\beta^S_0(N_1)}\oplus\left(\frac{M_2}{\m_TM_2}\right)^{\beta^S_1(N_1)}\\
&\cong\Tor^S_1(N_1,k)^{\beta^T_0(M_2)}\oplus\left(\frac{N_1}{\m_SN_1}\right)^{\beta^T_1(M_2)}.
\end{align*}
\end{lem}

\begin{proof}
We verify the first isomorphism. The others are obtained similarly.
Compute $\Tor^R_1(M_1,N_1)$ by tensoring the sequence~\eqref{eq160316a} with $N_1$.
\begin{equation}\label{eq160316b}
N_1^{\beta_2}\oplus N_1^{c_1\beta_1}\oplus N_1^{b_1c_1\beta_0}\oplus N_1^{c_2\beta_0}
\xra{\left[\begin{smallmatrix}\wti{d_2}' & \widehat{g_1}'&0&0 \\ 0&0&\widehat{f_1}' & \ol{g_2}' \end{smallmatrix}\right]}
N_1^{\beta_1}\oplus N_1^{c_1\beta_0}
\xra{\left[\wti{d_1}' \ \ \ol{g_1}'\right]}
N_1^{\beta_0}
\end{equation}
Since $N_1$ is an $S$-module, it is annihilated by $\m_T$,
so the fact that the entries of $\widehat{g_1}$ are in $\m_T$ implies that $\widehat{g_1}'=0$,
and similarly for $\overline{g_1}'$ and $\overline{g_2}'$.
Thus, the complex~\eqref{eq160316b} has the following form.
\begin{equation*}
N_1^{\beta_2}\oplus N_1^{c_1\beta_1}\oplus N_1^{b_1c_1\beta_0}\oplus N_1^{c_2\beta_0}
\xra{\left[\begin{smallmatrix}\wti{d_2}' & 0&0&0 \\ 0&0&\widehat{f_1}' & 0 \end{smallmatrix}\right]}
N_1^{\beta_1}\oplus N_1^{c_1\beta_0}
\xra{\left[\wti{d_1}' \ \ 0\right]}
N_1^{\beta_0}
\end{equation*}
From this, it follows that we have
\begin{align*}
\Tor^R_1(M_1,N_1)
&\cong \HH\left(N_1^{\beta_2}\xra{\wti{d_2}'}N_1^{\beta_1}\xra{\wti{d_1}'}N_1^{\beta_0}\right)\oplus
	\HH\left(N_1^{b_1c_1\beta_0}\xra{\widehat{f_1}'}N_1^{c_1\beta_0}\xra{0}N_1^{\beta_0}\right) \\
&\cong\Tor^S_1(M_1,N_1)\oplus\left(\frac{N_1}{\m_SN_1}\right)^{c_1\beta_0}
\end{align*}
as desired.
\end{proof}

The next two lemmas are essentially applications of the previous one, for use in our first main theorem.


\begin{lem}\label{lem160316b}
Let $M_1$, $N_1$ be  finitely generated $S$-modules.
If $\Tor^R_1(M_1,N_1)=0$, then $M_1=0$ or $N_1=0$.
\end{lem}

\begin{proof}
Assume that $M_1\neq 0$. Hence, $\beta^S_0(M_1)\neq 0$. 
Also, the assumption $T\neq k$ implies that $\beta^T_1(k)\neq 0$.
From Lemma~\ref{lem160316a}, we have
$$0=
\Tor^R_1(M_1,N_1)
\cong\Tor^S_1(M_1,N_1)\oplus\left(\frac{N_1}{\m_SN_1}\right)^{\beta^T_1(k)\beta^S_0(M_1)}
$$
so $N_1=\m_SN_1$. Now by Nakayama's Lemma we conclude that $N_1=0$.
\end{proof}

\begin{lem}\label{lem160316c}
Let $M_1$, $N_1$ be first syzygies over $S$ of finitely generated $S$-modules, and let $M_2$, $N_2$ be
first syzygies over $T$ of finitely generated $T$-modules.
Assume that $\depth(S)=0$ or $\depth(T)=0$,
and set $M:=M_1\oplus M_2$ and $N:=N_1\oplus N_2$.
Assume that $\Tor^R_1(M,N)=0$; then $M=0$ or $N=0$.
\end{lem}

\begin{proof}
Assume that $M_1\oplus M_2=M\neq 0$. We need to show that $N=0$.
By symmetry, assume further that $M_1\neq 0$.
The assumption $\Tor^R_1(M,N)=0$
implies that $\Tor^R_1(M_1,N_1)=0$, so Lemma~\ref{lem160316b} implies that $N_1=0$.

Suppose by way of contradiction that $N_2\neq 0$.
Another application of the Tor-vanishing assumption, using Lemma~\ref{lem160316a}, implies that
\begin{align*}
0
&=\Tor^R_1(M_1,N_2)\\
&\cong\Tor^T_1(k,N_2)^{\beta^S_0(M_1)}\oplus\left(\frac{N_2}{\m_TN_2}\right)^{\beta^S_1(M_1)}\\
&\cong\Tor^S_1(M_1,k)^{\beta^T_0(N_2)}\oplus\left(\frac{M_1}{\m_SM_1}\right)^{\beta^T_1(N_2)}.
\end{align*}
Since $\beta^S_0(M_1)\neq 0\neq \beta^T_0(N_2)$, we conclude that $\Tor^T_1(k,N_2)=0=\Tor^S_1(M_1,k)$.
Hence,  $M_1\neq 0$ is free over $S$ and $N_2\neq 0$ is free over $T$,
so $S$ is a summand of $M_1$ and $T$ is a summand of $N_2$.
Since $M_1$ is a syzygy over $S$, we have $S\subseteq M_1\subseteq\m_S S^m$ for some $m\geq 1$.
If $\depth(S)=0$, this is impossible. Indeed, this implies that there is an element $t\in \m_S S^m$ such that the map
$S\to \m_S S^m$ given by $s\mapsto st$ is a monomorphism; but the socle $\soc(S)\neq 0$ is contained in the kernel of this map.

On the other hand, the fact that $N_2$ is a syzygy over $T$ yields a contradiction in the case $\depth(T)=0$.
\end{proof}

\begin{disc}\label{disc160317a}
The
following computation
$$\Tor^R_1(S,T)=\Tor^R_1(R/\m_T,R/\m_S)\cong(\m_S\cap\m_T)/(\m_S\m_T)=0$$
shows  the necessity of the syzygy assumptions in the previous result,
since $S,T\neq 0$.
The isomorphism here is standard.
\end{disc}

The following result contains Theorem~\ref{thm160319a}\eqref{thm160319a1} from the introduction.

\begin{thm}\label{thm160316a}
Let $M$, $N$ be finitely generated modules  over the fiber product $R$.
If  $\depth(S)=0$ or $\depth(T)=0$,
then the following conditions are equivalent.
\begin{enumerate}[\rm(i)]
\item \label{thm160316a1}
$\Tor^R_i(M,N)=0$ for $i\gg 0$
\item \label{thm160316a2}
$\Tor^R_i(M,N)=0$ for all $i\geq 1$
\item \label{thm160316a3}
$\Tor^R_i(M,N)=0$ for some $i\geq 5$
\item \label{thm160316a4}
$\pd_R(M)<\infty$ or $\pd_R(N)<\infty$
\item \label{thm160316a5}
$M$ or $N$ is  $R$-free
\end{enumerate}
\end{thm}

\begin{proof}
Assume for this paragraph that there exists
an integer $i\geq 5$ such that $\Tor^R_i(M,N)=0$; we prove that $M$ or $N$ is  $R$-free.
Let $M'$ be the second syzygy of $M$, and let $N'$ be the $(i-3)$rd syzygy of $N$.
Dimension-shift to conclude that $\Tor^R_{1}(M',N')\cong\Tor^R_{3}(M,N')\cong\Tor^R_{i}(M,N)=0$.
The assumption $i\geq 5$ implies $i-3\geq 2$.
Thus, Fact~\ref{fact160316b} implies that $M'=M'_1\oplus M'_2$ and $N'=N'_1\oplus N'_2$ where
$M'_1$, $N'_1$ are (finitely generated) first syzygies over $S$, and
$M'_2$, $N'_2$ are (finitely generated) first syzygies over $T$.
Thus, Lemma~\ref{lem160316c} implies that $M'=0$ or $N'=0$. In light of the construction of $M'$ and $N'$, it follows that
$\pd_R(M)\leq 1<\infty$ or $\pd_R(N)\leq i-4<\infty$.
By equation~\eqref{eq160315a} from the introduction, we have $\depth(R)=0$.
Thus, by the Auslander-Buchsbaum formula,
we have
$\pd_R(M)=0$ or $\pd_R(N)=0$.

Using the previous paragraph, for $n= \text{i}, \ldots, \text{iv}$ we have $(n)\implies\eqref{thm160316a5}$. The converses of these implications are routine.
\end{proof}

Theorem~\ref{thm160316a} gives the following slightly weaker version of
a result of Nasseh and Yoshino~\cite[Theorem 3.1]{nasseh:oeire}; see also~\cite[Proposition 5.2]{avramov:htecdga}.
This result uses the ``trivial extension'' $S\ltimes k$; as an additive abelian group, this is $S\oplus k$, and multiplication
is given by the formula $(s,x)(t,y):=(st,sy+tx)$.
One shows readily that there is an isomorphism $S\ltimes k\cong S\times_k(k[x]/(x^2))$.
Our result is slightly weaker than Nasseh and Yoshino's result because that result only requires $i\geq 3$;
note that they also give an example showing that this range of $i$-values is optimal in their setting.

\begin{cor}\label{cor080915a}
Assume that $M$ and $N$ are non-zero finitely generated $S\ltimes k$-modules such that $\Tor_{i}^{S\ltimes k} (M,N) = 0$ for some $i\geq 5$;
then $M$ or $N$ is free over $S\ltimes k$.
\end{cor}

Here is an example showing that in the general setting of Theorem~\ref{thm160316a},
vanishing of $\Tor^R_3(M,N)$ is not enough to guarantee
$\pd_R(M)<\infty$ or $\pd_R(N)<\infty$, in contrast to the special case of~\cite[Theorem 3.1]{nasseh:oeire}.
Neither are two sequential vanishings $\Tor^R_2(M,N)=0=\Tor^R_{3}(M,N)$ enough.
(See also Theorem~\ref{thm160316b} below.)
In the example, we use the following straightforward fact: If $0\neq s\in\m_S$ and $0\neq t\in\m_T$, then
$\ann_R(s+t)=\ann_S(s)\oplus\ann_T(t)$.

\begin{ex}\label{ex160317a}
Consider the local artinian rings $S:=k[U,V]/(U^2,V^2)$ and $T:=k[X,Y]/(X^2,Y^2)$, and set $R:=S\times_kT$ as usual.
Use lower-case letters $u,v,x,y$ to represent the residues of the variables $U,V,X,Y$ in $R$ and in the respective rings $S$ and $T$.
With $M=R/(u+x)R$ and $N=R/(v+y)R$, we claim that $\Tor^R_i(M,N)=0$ if and only if $i=2$ or $3$.

Note that $M$ and $N$ are not free over $R$, since they have non-trivial annihilators.
Since $\depth(R)=0$, the Auslander-Buchsbaum formula implies that $M$ and $N$ have infinite projective dimension.
By Theorem~\ref{thm160316a}, it follows that $\Tor^R_i(M,N)\neq 0$ for all $i\geq 5$.
In addition, we have $\Tor^R_0(M,N)\cong \Otimes MN\cong R/(u+x,y+v)R\neq 0$.
Thus, it remains to show that $\Tor^R_1(M,N)\neq 0=\Tor^R_2(M,N)=\Tor^R_3(M,N)\neq \Tor^R_4(M,N)$.
To this end, we note the following straightforward computation.
\begin{equation}
(uS\oplus xT)\cap(vS\oplus yT)=uvS\oplus xyT=(uS\oplus xT)\cdot(vS\oplus yT) \label{eq160317b}
\end{equation}
Next, we explain the first equality in the following display
\begin{equation}
(u+x)R\cap(v+y)R=uvS\oplus xyT\neq(uv+xy)R=(u+x)R\cdot(v+y)R.\label{eq160317a}
\end{equation}
Here, the containment $(u+x)R\cap(v+y)R\subseteq uv S\oplus xyT$
follows from~\eqref{eq160317b} since we have $(u+x)R\subseteq (uS\oplus xT)$
and $(v+y)R\subseteq(vS\oplus yT)$. The reverse containment follows from the equalities
$v(u+x)=uv=u(v+y)$ and $x(v+y)=xy=y(u+x)$.
(The other steps in~\eqref{eq160317a} are even more straightforward.)
Similarly, we have
\begin{equation}
(uS\oplus xT)\cap(v+y)R=(uS\oplus xT)\cdot(v+y)R.\label{eq160317c}
\end{equation}
Now we show how this helps us to compute $\Tor^R_i(M,N)$.
We make repeated use of the formula $\Tor^R_1(R/I,R/J)\cong(I\cap J)/IJ$, first for $i=1$:
$$\Tor^R_1(M,N)=\Tor^R_1(R/(u+x)R,R/(v+y)R)\cong\frac{(u+x)R\cap(v+y)R}{(u+x)R\cdot(v+y)R}\neq 0$$
where the non-vanishing is from~\eqref{eq160317a}.
Similarly, as the paragraph preceding this example shows
$\ann_R(u+x)=uS\oplus xT$,
the next computation follows by dimension-shifting and using~\eqref{eq160317c}.
$$\Tor^R_2(M,N)=\Tor^R_1(R/(uS\oplus xT),R/(v+y)R)\cong\frac{(uS\oplus xT)\cap(v+y)R}{(uS\oplus xT)\cdot(v+y)R}= 0$$
Similarly, from~\eqref{eq160317b} we have $\Tor^R_3(M,N)=0$.
Thus, it remains to show that $\Tor^R_4(M,N)\neq 0$.
To this end, again by dimension-shifting we have
\begin{align*}
\Tor^R_4(M,N)
&\cong\Tor^R_2(R/(uS\oplus xT),R/(vS\oplus yT))\\
&\cong\Tor^R_1(uS\oplus xT,R/(vS\oplus yT))\\
&\cong\Tor^R_1(uS,R/(vS\oplus yT))
\oplus\Tor^R_1(xT,R/(vS\oplus yT)).
\intertext{So, it suffices to show that $\Tor^R_1(uS,R/(vS\oplus yT))\neq 0$, which we show next.}
\Tor^R_1(uS,R/(vS\oplus yT))
&\cong\Tor^R_1(R/\ann_R(u),R/(vS\oplus yT))\\
&\cong\Tor^R_1(R/(uS\oplus\m_T),R/(vS\oplus yT))\\
&\cong\frac{(uS\oplus\m_T)\cap(vS\oplus yT)}{(uS\oplus\m_T)\cdot(vS\oplus yT)}\\
&=\frac{uvS\oplus yT}{uvS\oplus y\m_T}\\
&\neq 0
\end{align*}
This completes the example.
\end{ex}

\begin{disc}\label{disc160317b}
As in the preceding example,
one can show that the ring $R=k[\![U,V]\!]/(UV)\times_kk[\![X,Y]\!]/(XY)$
with the modules $M=R/(u+x)R$ and $N=R/(v+y)R$ satisfy
$\Tor^R_i(M,N)=0$ if and only if $i=1$ or 3.
(Note that this ring fits in the context  of Theorem~\ref{thm160316b} below.)
\end{disc}

Next, we work on the case where $S$ and $T$ both have positive depth.

\begin{lem}\label{lem160316d}
Let $M_1$, $N_1$ be first syzygies over $S$ of finitely generated $S$-modules, and let $M_2$, $N_2$ be
first syzygies over $T$ of finitely generated $T$-modules.
Set $M:=M_1\oplus M_2$ and $N:=N_1\oplus N_2$.
Assume that $\Tor^R_1(M,N)=0=\Tor^R_2(M,N)$; then $M=0$ or $N=0$.
\end{lem}

\begin{proof}
As in the beginning of the proof of Lemma~\ref{lem160316c},
we assume that $M_1\neq 0$, and conclude that $N_1=0$.
And we suppose by way of contradiction that $N_2\neq 0$.
An application of (a symmetric version of) Lemma~\ref{lem160316b} implies that $M_2=0$.

By Fact~\ref{fact160316a},
a minimal $R$-free presentation of $M_1=M$ has the following form:
\begin{equation*}
R^{\beta_1}\oplus R^{c_1\beta_0}
\xra{\left[\wti{d_1} \ \ \ol{g_1}\right]}
R^{\beta_0}.
\end{equation*}
In particular, we have
\begin{equation}\label{eq160316b'}
\im(\ol{g_1})=\im(g_1)R^{\beta_0}=\m_TR^{\beta_0}\neq 0
\end{equation}
where the non-vanishing  follows from the assumption  $T\neq k$.

Next, set $M'=\im[\wti{d_1} \ \ \ol{g_1}]\subseteq R^{\beta_0}$, i.e., $M'$ is the first syzygy of $M$ in the above minimal presentation.
Since $M$ is a first syzygy, the module $M'$ is a second syzygy,
and Fact~\ref{fact160316b} implies that $M'\cong M'_1\oplus M'_2$ where $\m_TM'_1=0=\m_SM'_2$.
Moreover, this Fact explains the first equality in the next display.
\begin{equation}\label{eq160316c}
M'_2=\im[\wti{d_1} \ \ \ol{g_1}]\cap(\m_TR^{\beta_0})\supseteq \im(\ol{g_1})\cap(\m_TR^{\beta_0})=\m_TR^{\beta_0}\neq 0
\end{equation}
The containment here is straightforward; the second equality and the non-vanishing are from~\eqref{eq160316b'}.

Dimension-shifting gives
$0=\Tor^R_2(M,N)
\cong\Tor^R_1(M',N)$.
The condition $M'_2\neq 0$ from~\eqref{eq160316c}, with Lemma~\ref{lem160316b} implies that $N_2=0$, contradicting the
supposition in the first paragraph of this proof.
In other words, we must have $N_2=0$. With the already established $N_1=0$, we conclude that $N=0$, as desired.
\end{proof}

Next, we establish Theorem~\ref{thm160319a}\eqref{thm160319a2} from the introduction.

\begin{thm}\label{thm160316b}
Let $M$ and $N$ be finitely generated modules over the fiber product $R$.
Then the following conditions are equivalent.
\begin{enumerate}[\rm(i)]
\item \label{thm160316b1}
$\Tor^R_i(M,N)=0$ for $i\gg 0$
\item \label{thm160316b2}
$\Tor^R_i(M,N)=0$ for all $i\geq 2$
\item \label{thm160316b3}
$\Tor^R_i(M,N)=0=\Tor^R_{i+1}(M,N)$ for some $i\geq 5$
\item \label{thm160316b4}
$\pd_R(M)<\infty$ or $\pd_R(N)<\infty$
\item \label{thm160316b5}
$\pd_R(M)\leq 1$ or $\pd_R(N)\leq 1$
\end{enumerate}
\end{thm}

\begin{proof}
Argue as in the proof of Theorem~\ref{thm160316a}, using Lemma~\ref{lem160316d}
in lieu of~\ref{lem160316c}.
\end{proof}

In the preceding result, the next remark  shows that even if one assumes that $\Tor^R_i(M,N)=0$ for all $i\geq 1$,
one cannot conclude that
$M$ or $N$ is free (unless, of course, one is in the setting of Theorem~\ref{thm160316a}).

\begin{disc}\label{ex160325a}
Indeed, assume that $S$ and $T$ have
positive depth, so we have $\depth(R)=1$ by equation~\eqref{eq160319a}.
Let $f\in \m_R$ be $R$-regular, and set $M:=R/fR$.
Then $\pd_R(M)=1$, so $\Tor^R_i(M,-)=0$ for all $i\geq 2$. By dimension-shifting,
if $N$ is a syzygy over $R$ of infinite projective dimension,
e.g., if $N$ is a non-principal ideal of $R$,
then $\Tor^R_i(M,N)=0$ for all $i\geq 1$, even though $M$ and $N$ are not free.
\end{disc}

In light of the preceding theorems and examples, we pose the following.

\begin{question}\label{q160317a}
Let $M$, $N$ be finitely generated modules over the fiber product~$R$.
If $\Tor^R_4(M,N)=0$, must one of the modules $M$, $N$ have finite projective dimension?
\end{question}

\section{Auslander's Depth Formula}\label{sec160325a}

In this section, we document some consequences of the above results for
Auslander's ``depth formula'' from~\cite[Theorem~1.2]{auslander:murlr}.
This subject has received considerable attention recently; see, e.g., work of
Araya and Yoshino~\cite{araya:rdf},
Christensen and Jorgensen~\cite{christensen:vth},
and Foxby~\cite{foxby:hdcm}.

For part of the proof of the first result of this section,
we work in the derived category $\catd(R)$ with objects equal to the $R$-complexes
(i.e., the chain complexes of $R$-modules) indexed homologically.
References for this include
Christensen, Foxby, and Holm~\cite{christensen:dcmca},
Hartshorne and Grothendieck~\cite{hartshorne:rad}, and
Verdier~\cite{verdier:cd, verdier:1}.
We say that an $R$-complex $X$ is \emph{homologically finite} if
the total homology module $\HH(X)=\bigoplus_{i\in\bbz}\HH_i(X)$ is finitely generated.
We set $\Ext^i_R(X,Y):=\HH_{-i}(\Rhom XY)$
and
$\depth_R(X)=\inf\{i\in\bbz\mid\Ext_R^i(k,X)\neq 0\}$.

\begin{thm}\label{cor160319a}
Let $M$, $N$ be non-zero finitely generated modules over the fiber product $R$ such that
$\Tor^R_i(M,N)=0$ for all $i\gg 0$. Set
$$q:=\max\{i\geq 0\mid\Tor^R_i(M,N)\neq 0\}.$$
Assume that at least one of the following conditions holds.
\begin{enumerate}[\rm(1)]
\item\label{cor160319a1}
$q=0$
\item\label{cor160319a2}
$\depth_R(\Tor^R_q(M,N))\leq 1$
\item\label{cor160319a3}
$\depth(S)=0$ or $\depth(T)=0$
\item\label{cor160319a4}
one of the modules $M$, $N$ is a syzygy of a finitely generated $R$-module
\item\label{cor160319a5}
one of the modules $M$, $N$ is $S$ or $T$
\item\label{cor160319a6}
$\depth(N)=0$ and $\pd_R(N)=\infty$
\end{enumerate}
Then we have $q\leq 1$ and
\begin{equation}\label{eq160319a}
\depth_R(M)+\depth_R(N)=\depth(R)+\depth_R(\Tor^R_q(M,N))-q.
\end{equation}
\end{thm}

\begin{proof}
Theorem~\ref{thm160316b} implies that $q\leq 1$ and, say, $\pd_R(M)\leq 1$.
Thus it remains to establish equation~\eqref{eq160319a} under any of the conditions~\eqref{cor160319a1}--\eqref{cor160319a6}.
In cases~\eqref{cor160319a1}--\eqref{cor160319a2} this is from~\cite[Theorem~1.2]{auslander:murlr}.
Theorem~\ref{thm160316a} shows that
\eqref{cor160319a3}$\implies$\eqref{cor160319a1}. In particular, we assume for the remainder of this proof that
$\depth(S),\depth(T)\geq 1$.

\eqref{cor160319a4}
In the case where $N$ is a syzygy of some finitely generated $R$-module $L$,
for $i\geq 1$ we have
$$\Tor^R_i(M,N)\cong\Tor^R_{i+1}(M,L)=0$$
since $\pd_R(M)\leq 1$, by dimension-shifting.
On the other hand, if $M$ is a syzygy of $L$, we have $\pd_R(L)\leq \depth(R)=1$ by Auslander-Buchsbaum,
so $M$ is free. In each case, we conclude that $q=0$, so we are done by case~\eqref{cor160319a1}.

\eqref{cor160319a5}
The assumption $\depth(S)\geq 1$ yields an $S$-regular element $f\in \m_S$.
It follows readily that $fR\cong fS\cong S$ so $S$ is a syzygy over $R$, namely, the first syzygy of $R/fR$.
Symmetrically, we see that $T$ is an $R$-syzygy, so this case follows from the preceding one.

\eqref{cor160319a6}
Assume that $\depth(N)=0$ and $\pd_R(N)=\infty$. Assume further that
$q=1=\pd_R(M)$ and $\depth(S),\depth(T)\geq 1$, otherwise we are in the situation of case~\eqref{cor160319a1} or~\eqref{cor160319a3}.
Thus, we have $\depth(R)=1$ by equation~\eqref{eq160315a} from the introduction, and $\depth_R(M)=0$ by the Auslander-Buchsbaum formula.
We need to show that
$$\depth_R(M)+\depth_R(N)=\depth(R)+\depth_R(\Tor^R_1(M,N))-1.$$
In light of our assumptions, this reduces to showing that
$\depth_R(\Tor^R_1(M,N))=0$.
That is, we need to show that $\m_R\in\ass_R(\Tor^R_1(M,N))$.

From~\cite[Lemma~2.1]{foxby:hdcm}, we have
\begin{align*}
\depth_R(\Lotimes MN)
&=\depth_R(M)+\depth_R(N)-\depth(R)\\
&=-1 \\
&=-\sup\{i\in\bbz\mid\Tor^R_i(M,N)\neq 0\}.
\end{align*}
Thus, according to~\cite[(1.6.6)]{christensen:scatac}, we have
$\m_R\in\ass_R(\Tor^R_1(M,N))$, as desired.
\end{proof}

Because of the strength of Theorems~\ref{thm160316a} and~\ref{thm160316b},
we were surprised to find the next examples which show that the depth formula~\eqref{eq160319a}
fails over the fiber product $R$ in the general case $q=1$ when $\depth(S),\depth(T)\geq 1$.

\begin{ex}\label{ex160319a}
Set $S:=k[\![U,V,W]\!]$ and $T:=k[\![X,Y,Z]\!]$ with prime ideals
$\p:=US$ and $\q:=(Y,Z)T$. We consider the fiber product $R$ and the $R$-modules
$M=R/(U+X)R$ and $N=S/\p\oplus T/\q$.
Since $U+X$ is $R$-regular, we have $\pd_R(M)=1$.
Using the minimal free resolution
$0\to R\xra{U+X}R\to M\to 0$
we have the first isomorphism in the next display
\begin{align*}
\Tor^R_1(M,N)
&\cong\ker\left(N\xra{U+X}N\right)\\
&\cong\ker\left(S/\p\xra{U+X}S/\p\right)\oplus\ker\left(T/\q\xra{U+X}T/\q\right)\\
&\cong\ker\left(S/\p\xra{U}S/\p\right)\oplus\ker\left(T/\q\xra{X}T/\q\right)\\
&\cong S/\p.
\end{align*}
The second isomorphism is from the definition of $N$,
and the third one follows from the vanishings $X(S/\p)=0=U(T/\q)$.
For the last isomorphism here, note that $X$ is $(T/\q)$-regular and that $U(S/\p)=0$.

From this, we have
$$\depth_R(\Tor^R_1(M,N))=\depth_R(S/\p)=2$$
and the next display is by construction
$$\depth_R(N)=\min\{\depth_R(S/\p),\depth_R(T/\q)\}=\min\{2,1\}=1.$$
As we have $\depth(R)=1$ by equation~\eqref{eq160315a} from the introduction, the fact that $U+X$ is $R$-regular implies
$\depth_R(M)=0$. From these facts, we find that
$$\depth_R(M)+\depth_R(N)=1<2=
\depth(R)+\depth_R(\Tor^R_1(M,N))-1
$$
so the depth formula~\eqref{eq160319a} fails here.
\end{ex}

\begin{ex}\label{ex160319b}
Set $S=k[\![u,x,y,z,a]\!]$ and $T=k[\![v]\!]$.
In~\cite[Example~2.9]{araya:rdf}, the authors construct an ideal $I\subseteq S$ such that the module $N:=S/I$
satisfies $\depth_S(N)=3$ and $\depth_S(\Tor^S_1(S/aS,N))=2$.
The module $S/aS$ has projective dimension 1 over $S$, and one checks easily that
$$
\depth_S(S/aS)+\depth_S(N)=7>6
=\depth(S)+\depth_S(\Tor^S_1(S/aS,N))-1.
$$
Let $R$ be the fiber product, as usual, and set $M=R/(a+v)R$.
By equation~\eqref{eq160315a} from the introduction, we have $\depth(R)=1$.
The element $a+v$ is $R$-regular, so $\pd_R(M)=1$. Also, one has
$\depth_R(N)=\depth_S(N)=3$. In the next display, the isomorphisms are routine,
and the equality is from the condition $vN=0$.
\begin{align*}
\Tor^R_1(M,N)
&\cong\ker(N\xra{a+v}N)
=\ker(N\xra aN)\cong\Tor^S_1(S/aS,N)
\end{align*}
We conclude that
$\depth_R(\Tor^R_1(M,N))=\depth_S(\Tor^S_1(S/aS,N))=2$, so
$$\depth_R(M)+\depth_R(N)=3>2
=\depth(R)+\depth_R(\Tor^R_1(M,N))-1
$$
so the depth formula~\eqref{eq160319a} fails here, in the opposite way from Example~\ref{ex160319a}.
\end{ex}

The next result shows that the previous example is, in a sense, minimal with respect to the
particular failure of the depth formula~\eqref{eq160319a}.

\begin{prop}\label{prop160319a}
Let $(A,\m_A)$ be a commutative local noetherian ring, and let $M$, $N$ be finitely generated
$A$-modules such that $\pd_A(M)=1$.
If $\depth_A(N)\leq 2$, then
$$\depth_A(M)+\depth_A(N)\leq\depth(A)+\depth_A(\Tor^A_1(M,N))-1.$$
\end{prop}

\begin{proof}
In the case $\depth_R(N)=0$, the desired inequality follows by
the Auslander-Buchsbaum formula for $M$.
So, we assume for the rest of the proof that $\depth_A(N)$ is 1 or 2.
Consider a minimal free resolution
$$0\to A^m\to A^n\to M\to 0$$
wherein $m,n\neq 0$ since $\pd_A(M)=1$.
It follows that there is an exact sequence
$$0\to \Tor^A_1(M,N)\to N^m\to N^n\to M\otimes_AN\to 0.$$
Break this into short exact sequences
$$0\to \Tor^A_1(M,N)\to N^m\to V\to 0\qquad\qquad 0\to V\to N^n\to M\otimes_AN\to 0$$
and apply the Depth Lemma; the assumption $1\leq\depth_A(N)\leq 2$ implies that
\begin{align*}
\depth_A(V)
&\geq\min\{\depth_A(N),\depth_A(M\otimes_AN)+1\}\geq 1\\
\depth_A(\Tor^A_1(M,N))
&\geq\min\{\depth_A(N),\depth_A(V)+1\}
\geq\depth_A(N).
\end{align*}
In light of the Auslander-Buchsbaum formula for $M$, which has projective dimension 1,
the preceding display yields the desired inequality.
\end{proof}

\section{Ext-Vanishing}\label{sec160319d}

The results of this section, like those of Section~\ref{sec160319b},
give counterpoints to the theme discussed in Section~\ref{sec160319a}: the conclusions here hold over the fiber product $R$,
whether or not they hold over $S$ or $T$.

It is worth noting here that much of the machinery we invoke in the proof of the next result
is developed in significant generality in the forthcoming~\cite{avramov:phcnr}.
Thus, we only sketch the proof.

\begin{prop}\label{prop160316a}
Let $M$ and $N$ be homologically finite complexes over the fiber product $R$
such that $\Ext_R^i(M,N)=0$ for $i\gg 0$, then $\pd_R(M)<\infty$ or $\id_R(N)<\infty$.
\end{prop}

\begin{proof}
First, note that if $X$, $Y$ are homologically finite $R$-complexes such that
$\Tor^R_i(X,Y)=0$ for $i\gg0$, then  $\pd_R(X)<\infty$ or $\pd_R(Y)<\infty$.
Indeed, take sufficiently high syzygies in minimal $R$-free resolutions of $X$ and $Y$.
These are finitely generated and Tor-independent,
so the conclusion follows from Theorem~\ref{thm160316b}.

Let $K$ denote the Koszul complex over $R$ on a finite generating sequence for $\m_R$.
It follows that the complex  $\Lotimes KN=K\otimes_RN$
is homologically finite and, furthermore, has homology annihilated by $\m_R$.
Thus, it has finite length total homology module.
Moreover, we have $\id_R(\Lotimes KN)<\infty$ if and only if $\id_R(N)<\infty$.

Let $E:=E_R(k)$ be the injective hull of $k$ over $R$.
Since $\HH(\Lotimes KN)$ has finite length, so does $\HH(\Rhom{\Lotimes KN}E)$;
in particular, $\Rhom{\Lotimes KN}E$ is homologically finite.
With the ``Hom-evaluation'' isomorphism
of Avramov and Foxby~\cite[Lemma~4.4(I)]{avramov:hdouc}
\begin{align*}
\Rhom{\Rhom M{\Lotimes KN}}{E}
&\simeq\Lotimes M{\Rhom {\Lotimes KN}E}
\end{align*}
the preceding paragraph implies
that $\pd_R(M)<\infty$ or $\pd_R(\Rhom {\Lotimes KN}E)<\infty$.
It follows readily that $\pd_R(M)<\infty$ or $\id_R(\Lotimes KN)<\infty$,
i.e., $\pd_R(M)<\infty$ or $\id_R(N)<\infty$,  as desired.
\end{proof}

Next, we document some consequences of the preceding proposition.

\begin{cor}\label{cor160316a}
Let $M$, $N$ be finitely generated modules over the fiber product $R$.
If $\Ext_R^i(M,N)=0$ for $i\gg 0$, then $\pd_R(M)\leq 1$ or $\id_R(N)\leq 1$.
\end{cor}

\begin{proof}
This follows from Proposition~\ref{prop160316a}, via the Auslander-Buchsbaum and Bass formulas,
since $\depth(R)\leq 1$ by equation~\eqref{eq160315a}.
\end{proof}

\begin{cor}\label{cor160317a}
Assume that $\depth(S)=0$ and that $S$ or $T$ is not artinian.
Let $M$, $N$ be non-zero finitely generated modules over the fiber product $R$.
If $\Ext_R^i(M,N)=0$ for $i\gg 0$, then $M$ is $R$-free.
\end{cor}

\begin{proof}
Equation~\eqref{eq160315a} shows that $\depth(R)=0$.
Thus, the preceding result (with the Auslander-Buchsbaum and Bass formulas)
shows that $M$ is free or $N$ is injective.
Since $R$ surjects onto $S$ and $T$, one of which is not artinian, we know that $R$ is not artinian,
so it does not have a finitely generated injective module. In particular, $N$ is not injective, so $M$ is free.
\end{proof}

The next result is immediate from the previous one.
It says that the fiber product $R$ has ``Ext-index'' at most $\depth(R)\leq 1$.
See equation~\eqref{eq160315a}.

\begin{cor}\label{cor160316c}
Let $M$, $N$ be finitely generated modules over the fiber product $R$ such that
$\Ext_R^i(M,N)=0$ for $i\gg 0$. Then $\Ext_R^i(M,N)=0$ for all $i>\depth(R)$.
\end{cor}

The next result says in particular that the fiber product $R$ satisfies the Auslander-Reiten Conjecture.
It contains Theorem~\ref{thm160319b} from the introduction, and it
follows, e.g., from the previous corollary and~\cite[Theorem~2.3]{christensen:asacvc}.

\begin{thm}\label{cor160316b}
Let $M$ be a finitely generated module over the fiber product $R$.
\begin{enumerate}[\rm(a)]
\item\label{cor160316b1}
If $\Ext_R^i(M,M\oplus R)=0$ for $i\gg 0$, then $\pd_R(M)\leq 1$.
\item\label{cor160316b1}
If $\Ext_R^i(M,M\oplus R)=0$ for $i\geq 1$, then $M$ is $R$-free.
\end{enumerate}
\end{thm}

The next result provides yet another cointerpoint the the theme from Section~\ref{sec160319a},
as one can easily construct examples where $S$ and $T$ have non-trivial semidualizing complexes (even modules);
recall that a homologically finite $R$-complex $C$ is \emph{semidualizing} if $\Rhom CC\simeq R$ in $\catd(R)$.

\begin{cor}\label{cor160316d}
The fiber product $R$ has at most two semidualizing complexes, up to shift-isomorphism, namely
$R$ and a dualizing $R$-complex (if $R$ has one).
\end{cor}

\begin{proof}
By definition, if $C$ is a semidualizing $R$-complex, then it is homologically finite such that $\Ext^i_R(C,C)=0$ for all $i\geq 1$.
Thus, Proposition~\ref{prop160316a} implies that $\pd_R(C)<\infty$ or $\id_R(C)<\infty$, that is,
$C\simeq\shift^nR$ for some $n\in\bbz$ (by a result of Christensen~\cite[Theorem~8.1]{christensen:scatac}) or $C$ is dualizing (by definition).
\end{proof}

It is not clear to us when the fiber product $R$ has a dualizing complex.
Of course, if it does, then so do the homomorphic images $S$ and $T$.
On the other hand, if $S$ and $T$ are complete, then so is $R$ by
a result of Grothendieck~\cite[(19.3.2.1)]{grothendieck:ega4-4}, so $R$
has a dualizing complex in this case.
Hence, we pose the following.

\begin{question}\label{q160316a}
If $S$ and $T$ admit dualizing complexes, must $R$ also admit one?
\end{question}

Next, we recall Auslander and Bridger's G-dimension~\cite{auslander:smt}.
A finitely generated $R$-module $G$ is \emph{totally reflexive} if
$\Hom_R(\Hom_R(G,R),R)\cong G$ and
$\Ext^i_R(G,R)=0=\Ext^i_R(\Hom_R(G,R),R)$ for all $i\geq 1$.
Every finitely generated free $R$-module is totally reflexive, so every finitely generated $R$-module $M$
has a resolution by totally reflexive modules.
If $M$ has a bounded resolution by totally reflexive modules, then the G-dimension of $M$ is the length of the
shortest such resolution.
The following result says, in the language of Takahashi~\cite{takahashi:grlr}, that our ring $R$ is ``G-regular''.

\begin{cor}\label{cor160325a}
Assume that the fiber product $R$ is not Gorenstein.\footnote{See~\cite[Section~3]{christensen:gmirlr} for a discussion of this assumption.}
Let $M$ be a finitely generated $R$-module.
Then one has $\pd_R(M)=\gdim_R(M)$.
In particular, $M$ is totally reflexive if and only if it is free.
\end{cor}

\begin{proof}
By~\cite[4.13]{auslander:smt}, it suffices to assume that $\gdim_R(M)<\infty$,
and prove that $\pd_R(M)<\infty$.
This assumption implies that we have $\Ext^i_R(M,R)=0$ for $i\gg 0$,
so Corollary~\ref{cor160316a} implies that $\pd_R(M)<\infty$ or $\id_R(R)<\infty$.
Since $R$ is not Gorenstein, it follows that $\pd_R(M)<\infty$, as desired.
\end{proof}

\section*{Acknowledgments}
We are grateful to Luchezar L.\ Avramov and and W.\ Frank Moore for useful conversations about this work.


\providecommand{\bysame}{\leavevmode\hbox to3em{\hrulefill}\thinspace}
\providecommand{\MR}{\relax\ifhmode\unskip\space\fi MR }
\providecommand{\MRhref}[2]{%
  \href{http://www.ams.org/mathscinet-getitem?mr=#1}{#2}
}
\providecommand{\href}[2]{#2}

\end{document}